\documentclass[a4paper]{amsart}

\usepackage{svninfo}

\usepackage{amsmath, amsthm, amsfonts, amssymb, latexsym, mathrsfs,stmaryrd}
\usepackage{mathtools}
\usepackage{accents}
\usepackage{cancel} 
\usepackage{hyperref}
\usepackage[utf8]{inputenc}
\usepackage{booktabs,colortbl}

\renewcommand{\today}{
  \ifcase\month\or
  January\or February\or March\or April\or May\or June\or
  July\or August\or September\or October\or November\or December\fi
  \space \number\year}

\newcommand\nmodels{\mathbin{\cancel{\models}}}

\newenvironment{fent}[1]{\buildrel\textstyle ^{#1}\over{\hbox{\vrule height-10pt depth5pt width-3pt}{\smash\models}}}

\newcommand\model[2]{\ensuremath{\fent{#1}_{\hbox{\vrule height-10pt depth5pt width-3pt}^{#2}}}}

\newenvironment{nfent}[1]{\buildrel\textstyle ^{#1}\over{\hbox{\vrule height-10pt depth5pt
width-3pt}{\smash\nmodels}}}

\newcommand\nmodel[2]{\ensuremath{\nfent{#1}_{\hbox{\vrule height-10pt depth5pt width-3pt}^{#2}}}}

\newenvironment{dotto}{\buildrel\textstyle.\over{\hbox{\vrule height3pt depth0pt width0pt}{\smash\to}}}

\newenvironment{dotv}{\buildrel\textstyle.\over{\hbox{\vrule height3pt depth0pt width0pt}{\smash\vee}}}

\newenvironment{dota}{\buildrel\textstyle.\over{\hbox{\vrule height-2pt depth0pt width0pt}{\smash\wedge}}}

\newenvironment{g}{G\"{o}del~}

\newenvironment{lo}{{\L}ukasiewicz~}

\newenvironment{ds}{\displaystyle}

\newtheorem{thm}{Theorem}[section]

\newtheorem{cor}[thm]{Corollary}
\newtheorem{lem}[thm]{Lemma}

\theoremstyle{definition}
\newtheorem{defn}[thm]{Definition}

\newtheorem{exa}[thm]{Example}

\theoremstyle{remark}
\newtheorem{rem}[thm]{Remark}

\DeclareSymbolFont{AMSb}{U}{msb}{m}{n}
\DeclareMathSymbol{\N}{\mathbin}{AMSb}{"4E}
\DeclareMathSymbol{\Z}{\mathbin}{AMSb}{"5A}
\DeclareMathSymbol{\R}{\mathbin}{AMSb}{"52}
\DeclareMathSymbol{\Q}{\mathbin}{AMSb}{"51}
\DeclareMathSymbol{\I}{\mathbin}{AMSb}{"49}
\DeclareMathSymbol{\C}{\mathbin}{AMSb}{"43}

\makeatletter

\def\dotminussym#1#2{%
  \setbox0=\hbox{$\m@th#1-$}%
  \kern.5\wd0%
  \hbox to 0pt{\hss\hbox{$\m@th#1-$}\hss}%
  \raise.8\ht0\hbox to 0pt{\hss$\m@th#1.$\hss}%
  \kern.5\wd0}

\makeatother

\title{On the compactness property of extensions of first-order G\"{o}del logic}

\author{Seyed Mohammad Amin Khatami}
\address{Department of Mathematics and Computer Science, Amirkabir University of Technology, Tehran, Iran}
\email{amin\_khatami@aut.ac.ir} \urladdr{}

\begin{document}
\maketitle
\begin{abstract}
We study three kinds of compactness in some variants of \g logic: compactness,
entailment compactness, and approximate entailment compactness.
For countable first-order underlying language we use the Henkin
construction to prove the compactness property of extensions of
first-order \g logic enriched by nullary connective or the Baaz's
projection connective. In the case of uncountable first-order language
we use the ultraproduct method to derive the compactness theorem.\\
\smallskip
\noindent\emph{Keywords:} Mathematical fuzzy logic; \g logic; Rational \g logic; Compactness theorem
\end{abstract}
\section{Introduction}
Compactness theorem is one of the most important theorems in classical first-order logic.
This theorem says that any finitely satisfiable theory is satisfiable. Certainly, this
property provides a procedure to find models of a theory whose
finite subsets have models. So, it could be considered as a foundation for
model theoretical studies of any logic. Due to the fact that the
model theory of mathematical fuzzy logic is still underdeveloped, study of
compactness property would be a topic of interest in the area of
mathematical fuzzy logic. In the case of t-norm based fuzzy logics and
their extensions, this is done by several authors
\cite{hajek98,baaz1998compact,navara2001compactness,gerla2001abstract,cintula2004compactness,
benust10,preining2003complete,cintula2005two,tpd2011,pourmahtavana2012}.

Among t-norm based fuzzy logics, three of them are quite important
(\g, \lo, and product logic). So, almost all studies around
compactness property are done for these triple.Note that various kinds
of compactness are available for t-norm based fuzzy logic, e.g., compactness
\cite{navara2001compactness,gerla2001abstract,cintula2004compactness,cintula2005two},
entailment compactness \cite{baaz1998compact,preining2003complete,cintula2005two}and $K$-compactness
\cite{cintula2005two,tpd2011,pourmahtavana2012} where $K$ is a closed subset
of standard truth value set $[0,1]$. The usual compactness
is the same as $\{1\}$-compactness. Let us remind that a logic
enjoys the entailment compactness if for every theory $T$ and sentence
$\varphi$, $T\models\varphi$ implies the existence of a finite
subset $T'$ of $T$ such that $T'\models\varphi$.

In first-order \g logic, different truth value sets cause
different results about compactness. A truth value set in general is taken
to be any linearly ordered Heyting algebra $\mathbb{D}$. The standard truth
value set is commonly assumed
to be a \g set which is a closet subset of $[0,1]$ containing $0$ and $1$.
The first-order \g logic
whose truth value set is a \g set $V$ is denoted by $\mathfrak{G}_V$.
Recently, all tree mentioned instance of compactness are studied
for \g set $\mathfrak{G}_V$ \cite{pourmahtavana2012,preining2003complete}.
Furthermore, \cite{preining2003complete} studies the extensions of
\g logic $\mathfrak{G}_V$ by $\Delta$ Baaz projection connective.

In \lo logic as well as its extension such as rational Pavelka logic (RPL)
and continuous first-order logic (CFO) the compactness theorem
is extensively studied in several frameworks
\cite{hajek98,cignoli2007lukasiewicz,esteva2007adding,pavelka79,tpd2011,benust10}. The
continuity of logical connectives of \lo logic
with respect to the usual order topology on $[0,1]$ is the main reason for
the compactness theorem to be held in these logics. By different methods
such as Henkin construction, Pavelka completeness, and ultraproduct
method the compactness theorem proved in these logics.

Study of the compactness property for extensions of \g logic is different from
two viewpoints. Firstly, the \g logic implication is not a continuous function
with respect to the usual order topology on \g sets. So,
the Pavelka method and ultraproduct method could not be used
directly in extensions of \g logic. However, a modification of these methods
may work here. Secondly, the corresponding algebras with respect to the
extensions of \g logics can not be embedded into the standard truth value
sets (\g sets) unless the algebras are at most countable. But, we
need such an embedding to prove the compactness theorem by the Henkin construction .
So, the Henkin construction only works for theories with at most countable first-order
underlying languages.

We consider two approaches to prove the compactness property in extensions of
\g logics. The first one is based on the Henkin construction, and so it
works only for theories with at most countable first-order underlying
languages. The other approach is based on the ultraproduct method.
We consider a metric on \g sets such that the logical connectives
of the corresponding extension of \g logic are continuous with respect
to the new metric.

In \lo logic if "$e$" is a similarity relation, then the interpretation
of "$1-e$" becomes a pseudometric. But, we have not a logical connective
such as "minus" in \g logic. However, if one considers a reverse semantical meaning
on truth value set, the interpretation of similarity relation will be a
pseudometric in any t-norm based fuzzy logic. Furthermore, assuming
such a semantic leads to obtain a pseudometric
on the corresponding algebras of the logic. Besides these two pseudometrics,
the continuity of logical connectives and also continuity of the
interpretation of function and predicate symbols are directly intelligible.
So, using the ultraproduct method motivates us to consider a reverse semantical
meaning on \g sets which we call it the metrically semantic of the logic. Thus,
$0$ stands for absolute truth while $1$ for absolute falsity. Anyway,
we present a translation of results for the everyday \g logic in the final section.

This paper is organized as follows. In the next section, we introduce the main
notions of extensions of \g logic such as logical connectives, metrically
semantic, satisfiability, and so forth. Section 3 studies the main
concept of the paper by studding different notions of
compactness in several kinds of extensions of \g logic.
Section 4 presents the notion of
ultrametric structure and prove the compactness property for
some variants of \g logics without any limitation on the size of the underlying
first-order language. In the last section, a translation of results
for the usual semantic (in which 0 stands for falsity) is given.
\section{Preliminaries}\label{section godel logic}
The logic that we will consider in this paper is the \g logic whose semantic is
based on \g sets, i.e, subsets of unite interval $[0,1]$ containing $0$ and $1$ and
closed under the standard order topology. Logical symbols of the first-order \g
logic are the usual connectives of classical
first-order logic $\{\wedge, \to, \bot\}$ together with the quantifiers $\{\forall, \exists\}$
and a countable set of variables.

We use a reverse semantical meaning on the set of truth values. Indeed, this assumption makes
the interpretation of similarity relation a pseudometric. So, semantically $0$ is the absolute truth
and $1$ is the absolute falsity of the truth value set.

When a \g set $V$ is considered as the set of truth values, we use
the notion $\mathfrak{G}_V$ for corresponding \g logic.
Enriching $\mathfrak{G}_V$ by a countable set of nullary connectives
$\bar{A}=\{\bar{r}: r\in A\subseteq V\setminus\{0,1\}\}$ leads to an extension of \g logic,
$\mathfrak{G}_{V,A}$. Observe that the nullary connective $\bar{1}$ is actually
$\bot$. Another extension of \g logic is obtained by adding the unary
connective $\Delta$. The corresponding
\g logics equipped by $\Delta$ are denoted by $\mathfrak{G}^\Delta_V$ or
$\mathfrak{G}^\Delta_{V,A}$, respectively. Lets take an abbreviation for some \g logics:
\begin{itemize}
\item $\mathfrak{G}_\mathbb{R}$: $V=[0,1]$ and $A=\emptyset$.
\item $\mathfrak{G}_\downarrow$: $V=[0,1]_\downarrow$ and $A=(0,1)_\downarrow$ where $[0,1]_\downarrow=\{\frac{1}{n}: n\in\mathbb{N}\}\cup\{0\}$ and
$(0,1)_\downarrow=[0,1]_\downarrow\setminus\{0,1\}$.
\item $\mathfrak{G}_n$: $V=\{r_1, ..., r_n\}\cup\{0,1\}$ and $A=\{r_1, ..., r_n\}$ where $0<r_1< r_2<...<r_{n-1}<r_n<1$.
\item $\mathfrak{G}^*_\downarrow$: $V=[0,1]$ and $A=(0,1)_\downarrow$.
\item $\mathfrak{G}^*_n$: $V=[0,1]$ and $A=\{r_1, ..., r_n\}$ where $0<r_1< r_2<...<r_{n-1}<r_n<1$.
\item $RGL$: $V=[0,1]$ and $A=(0,1)\cap\mathbb{Q}$.
\end{itemize}
Within this paper, we assume that $\mathcal{L}$ is a first-order language.
$\mathcal{L}$-terms and $\mathcal{L}$-formulas are constructed as in classical first-order logic.
Basic notions of free and bound variable, $\mathcal{L}$-sentence and $\mathcal{L}$-theory
are defined as usual. In particular, note that $\bar{r}$ is an $\mathcal{L}$-sentence in
$\mathfrak{G}_{V,A}$ for each $r\in A$. The set of $\mathcal{L}$-formulas
and $\mathcal{L}$-sentences are denoted by
$Form(\mathcal{L})$ and $Sent(\mathcal{L})$, respectively. When
there is no danger of confusion, we may omit the prefix $\mathcal{L}$
and simply write a term, formula, etc.
\begin{defn}
For a given language $\mathcal{L}$, an
$\mathcal{L}$-\emph{structure} $\mathcal{M}$ in \g logic $\mathfrak{G}_{V,A}$ is a
nonempty set M called the universe of $\mathcal{M}$ together with:
  \begin{enumerate}
  \item for any n-ary predicate symbol $P$ of $\mathcal{L}$, a function
  $P^{\mathcal{M}}:M^n\to V$,
  \item for any n-ary function symbol $f$ of $\mathcal{L}$, a function $f^{\mathcal{M}}:M^n\to M$,
  \item for any constant symbol c of $\mathcal{L}$, an element $c^{\mathcal{M}}$ in the universe of $\mathcal{M}$.
  \end{enumerate}
When the underlying language is clear, $\mathcal{M}$ is called a structure.
\end{defn}
For each $\alpha\in\mathcal{L}$, $\alpha^\mathcal{M}$ is called the
\emph{interpretation} of $\alpha$ in $\mathcal{M}$. The
interpretation of terms is defined as follows.
\begin{defn}
For every n-tuple variable $\bar{x}$ and every term $t(\bar{x})$,the interpretation of $t(\bar{x})$ in
$\mathcal{M}$ is a function $t^\mathcal{M}:M^n\to M$ such that
\begin{enumerate}
\item if $t(\bar{x})=x_i$ for $1\le i\le n$, then $t^\mathcal{M}(\bar{a})=a_i$,
\item if $t(\bar{x})=c$ then $t^\mathcal{M}(\bar{a})=c^{\mathcal{M}}$,
\item if $t(\bar{x})=f(t_1(\bar{x}),...,t_n(\bar{x}))$ then $t^\mathcal{M}(\bar{a}) = f^\mathcal{M}(t_1^\mathcal{M}(\bar{a}),...,t_n^\mathcal{M}(\bar{a}))$.
\end{enumerate}
\end{defn}
Considering $0$ as the absolute truth makes some changes in some semantical issues.
For example, the interpretation of $\varphi\wedge\psi$ in a structure is absolutely true whenever
the interpretation of both of them are absolutely true, i.e, the maximum of their interpretations
must be absolutely true. The interpretation of formulas is defined as follows.
\begin{defn}
The interpretation of a formula $\varphi(\bar{x})$ in an $\mathcal{L}$-structure $\mathcal{M}$ in
the \g logic $\mathfrak{G}_{V,A}$ ($\mathfrak{G}^\Delta_{V,A}$) is a function $\varphi^\mathcal{M}:M^n\to V$ which is
inductively determined as follows.
\begin{enumerate}
\item $\bot^\mathcal{M}=1$, and for each $r\in A\cup\{0\}$, $\bar{r}^\mathcal{M}=r$.
\item For every n-ary predicate symbol $P$, $P^\mathcal{M}(t_1(\bar{a}),...,t_n(\bar{a}))=P^\mathcal{M} (t_1^\mathcal{M}(\bar{a}),...,t_n^\mathcal{M}(\bar{a}))$.
\item $(\varphi \wedge\psi)^\mathcal{M}(\bar{a}) =\max\{\varphi^\mathcal{M}(\bar{a}),\psi^\mathcal{M}(\bar{a})\}$.
\item $(\varphi \to \psi)^\mathcal{M}(\bar{a}) =\varphi^\mathcal{M}(\bar{a})\dotto\psi^\mathcal{M}(\bar{a})$, where
$x\dotto y=\left\{\begin{array}{cc}
0&x\ge y,\\
y&x<y.
\end{array}\right.$
\item If $\varphi(\bar{x})=\forall y\ \psi(y,\bar{x})$ then $\varphi^\mathcal{M}(\bar{a})=\displaystyle\sup_{b\in M}\{\psi^\mathcal{M}(b,\bar{a})\}$.
\item If $\varphi(\bar{x})=\exists y\ \psi(y,\bar{x})$ then $\varphi^\mathcal{M}(\bar{a})=\displaystyle\inf_{b\in M}\{\psi^\mathcal{M}(b,\bar{a})\}$.
\item (Only for $\mathfrak{G}^\Delta_{V,A}$) $(\Delta(\varphi))^\mathcal{M}(\bar{a})=\left\{
\begin{array}{ll}
0&\varphi^\mathcal{M}(\bar{a})=0,\\
1&otherwise.
\end{array}\right.$
\end{enumerate}
\end{defn}
Observe that since $V$ is a closed subset of $[0,1]$, all infima and suprema exist. One can consider an
abbreviation for compound connectives $\neg, \vee, \Rightarrow$ and $\leftrightarrow$.
\begin{itemize}
  \item $\neg\varphi:=\varphi\to\bot$.
  \item $\varphi\vee\psi:=((\varphi\to\psi)\to\psi)\wedge((\psi\to\varphi)\to\varphi)$.
  \item $\varphi\Rightarrow\psi:= (\psi\to\varphi)\to\psi$.
  \item $\varphi\leftrightarrow\psi:=(\varphi\to\psi)\wedge(\psi\to\varphi)$.
\end{itemize}
The interpretations of formulas including these new connectives can be computed as follows.
\begin{itemize}
  \item $\neg\varphi^\mathcal{M}(\bar{a})=\left\{\begin{array}{cc}
0&\varphi^\mathcal{M}(\bar{a})=1,\\
1&\varphi^\mathcal{M}(\bar{a})<1.
\end{array}\right.$
  \item $(\varphi \vee \psi)^\mathcal{M}(\bar{a})=\min\{\varphi^\mathcal{M}(\bar{a}),\psi^\mathcal{M}(\bar{a})\}$.
  \item $(\varphi \Rightarrow \psi)^\mathcal{M}(\bar{a}) =\left\{\begin{array}{cc}
0&\varphi^\mathcal{M}(\bar{a})>\psi^\mathcal{M}(\bar{a})>0,\\
\psi^\mathcal{M}(\bar{a})&\varphi^\mathcal{M}(\bar{a})\le\psi^\mathcal{M}(\bar{a}).
\end{array}\right.$
  \item $(\varphi \leftrightarrow \psi)^\mathcal{M}(\bar{a}) =
d_{max}(\varphi^\mathcal{M}(\bar{a}),\psi^\mathcal{M}(\bar{a}))$, where
$d_{max}(x,y)=\left\{\begin{array}{cc}
0&x=y,\\
\max\{x,y\}&x\ne y.
\end{array}\right.$
\end{itemize}
\begin{defn}
Let $\varphi(\bar{x})$ be
an $\mathcal{L}$-formula and $T$ be an $\mathcal{L}$-theory.
\begin{enumerate}
\item An $\mathcal{L}$-structure $\mathcal{M}$ is called a model of $\varphi(\bar{x})$,
if there is $\bar{a}\in M^n$ such that $\varphi^\mathcal{M}(\bar{a})=0$. In such a case, we write $\mathcal{M}\models\varphi(\bar{a})$.
\item $\varphi(\bar{x})$ is called a satisfiable formula if there is an
$\mathcal{L}$-structure $\mathcal{M}$ which models $\varphi(\bar{x})$.
\item If an $\mathcal{L}$-structure $\mathcal{M}$ models all sentences of $T$,
we call $T$ a \emph{satisfiable} theory an write $\mathcal{M}\models T$.
\item $T$ is called finitely satisfiable if every finite subset of $T$ has a model.
\item For an $\mathcal{L}$-sentence $\varphi$ we say that $T$ \emph{entails} $\varphi$,
$T\models\varphi$, if every model of $T$ models $\varphi$. We write $T\model{f}{}\varphi$ if there exists
a finite subset $S$ of $T$ so that $S\models\varphi$. If there is no finite subset
$S$ of $T$ such that $S\models\varphi$, we write $T\nmodel{f}{}\varphi$. We use $\models\varphi$
instead of $\emptyset\models\varphi$.
\end{enumerate}
\end{defn}
For any \g set $V$ and $A\subseteq V$, the axioms of the \g logic
$\mathfrak{G}_{V,A}$ are the axioms of first-order \g logic \cite{hajek98} together with
the book-keeping axioms listed in Table \ref{axioms}.
\begin{table}[ht]
\caption{Axioms and Rules of $\mathfrak{G}_{V,A}$}
\label{axioms}
\centering
\begin{tabular}{l l}
\toprule
\rowcolor[gray]{.8}
Axioms of fist-order \g logic&\\
\midrule
(G1) $(\varphi \rightarrow \psi)\rightarrow ((\psi\rightarrow \chi)\rightarrow(\varphi\rightarrow
\chi))$&\\
(G2) $(\varphi\wedge\psi)\rightarrow \varphi$&\\
(G3) $(\varphi\wedge\psi)\rightarrow(\psi\wedge\varphi)$&\\
(G4) $\varphi\rightarrow(\varphi\wedge\varphi)$&\\
(G5) $(\varphi\rightarrow(\psi\rightarrow\chi))\leftrightarrow((\varphi\wedge\psi)\rightarrow
\chi)$&\\
(G6) $((\varphi\rightarrow\psi)\rightarrow\chi)\rightarrow(((\psi\rightarrow\varphi)\rightarrow\chi)\rightarrow\chi)$&\\
(G7) $\bar{1}\rightarrow\varphi$&\\
(G$\forall$1) $(\forall x\,\varphi(x))\to\varphi(t)$ &($t$ substitutable for
$ x$ in $\varphi(x)$)\\
(G$\forall$2) $\big(\forall x\,(\psi\to\varphi(x))\big)\to\big(\psi\to(\forall
x\,\varphi(x))\big)$ &($x$ not free in $\psi$)\\
(G$\forall$3) $\big(\forall x\,(\psi\vee\varphi(x))\big)\to\big(\psi\vee(\forall
x\,\varphi(x))\big)$ &($x$ not free in $\psi$)\\
(G$\exists$1) $\varphi(t)\to(\exists x\,\varphi(x))$ &($t$ substitutable for $ x$ in $\varphi(x)$)\\
(G$\exists$2) $\big(\forall x\,(\varphi(x)\to\psi)\big)\to\big((\exists x\,\varphi(x))
\to\psi\big)$   &($x$ not free in $\psi$)\\
\midrule
\rowcolor[gray]{.8}
Book-keeping axioms for nullary connectives& \\
\midrule
(RG1) $\bar{r}\wedge\bar{s}\leftrightarrow\overline{\max\{r,s\}}$&\\
(RG2(a)) $\bar{r}\to\bar{s}$ &(for $r\ge s$)\\
(RG2(b)) $(\bar{r}\to\bar{s})\leftrightarrow\bar{s}$ &(for $r<s$)\\
(RG3) $\neg\neg\bar{r}$ &(for $r<1$)\\
\midrule
\rowcolor[gray]{.8}
Rules&\\
\midrule
(Mp) $\varphi, (\varphi\to\psi)\vdash\psi$&\\
(Gen) $\varphi\vdash\forall x\,\varphi$&\\
\bottomrule
\end{tabular}
\end{table}
\begin{defn}\label{proof}
An $\mathcal{L}$-sentence $\varphi$ is proved by an $\mathcal{L}$-theory $T$, $T\vdash\varphi$, whenever there
is a finite sequence $\{\varphi_i\}_{i=1}^n$ of $\mathcal{L}$-sentences such that:
\begin{itemize}
\item for each $1\le i\le n$ either $\varphi_i\in T$ or $\varphi_i$ is an axiom or it is
followed by rules from axioms and other $\varphi_j$'s for $1\le j< i$.
\item $\varphi_n=\varphi$.
\end{itemize}
We write $\vdash\varphi$ whenever $\emptyset\vdash\varphi$. $T$ is called a consistent theory if $T\nvdash\bot$.
\end{defn}
Note that if $A\ne\emptyset$ then for any $0<r<1$, $T=\{\bar{r}\}$ is a consistent theory
in $\mathfrak{G}_{V,A}$. However, $T$ is not a satisfiable theory. In the next section we
introduced the notion of strongly consistency which is equivalent to the notion of
satisfiability in some extensions of \g logics. The deduction theorem follows easily.
\begin{thm}
In the \g logic $\mathfrak{G}_{V,A}$,
for an $\mathcal{L}$-theory $T$ and $\mathcal{L}$-sentences $\varphi$ and $\psi$,
\begin{center}
$T\cup\{\varphi\}\vdash\psi$ if and only if $T\vdash\varphi\to\psi$.
\end{center}
\end{thm}
Obviously if $T\vdash\varphi$, then $T\models\varphi$ and also $T\model{f}{}\varphi$.
In spite of first-order logic, the concept of proof does not coincide completely with the concept of
finitely entailment in \g logics enriched by nullary connectives.
\begin{exa}\label{fail strong completeness} Let $A\ne\emptyset$.
One could easily verify that in the \g logic $\mathfrak{G}_{[0,1],A}$  if $\mathcal{L}=\{\rho\}$
where $\rho$ is a nullary predicate symbol and $r\in A\setminus\{0,1\}$,
then $\neg\neg \rho\to\bar{r}\models\neg \rho$ while
$\neg\neg \rho\to\bar{r}\nvdash\neg \rho$.
\end{exa}
\begin{rem}
For \g logics enriched by $\Delta$ connective, there are some additional axioms and rules.
\begin{itemize}
\item[$\Delta 1$)] $\Delta\varphi\vee\neg\Delta\varphi$.
\item[$\Delta 2$)] $\Delta(\varphi\vee\psi)\to(\Delta\varphi\vee\Delta\psi)$.
\item[$\Delta 3$)] $\Delta\varphi\to\varphi$.
\item[$\Delta 4$)] $\Delta\varphi\to\Delta\Delta\varphi$.
\item[$\Delta 5$)] $\Delta(\varphi\to\psi)\to(\Delta\varphi\to\Delta\psi)$.
\item[$\Delta$R)] $\varphi\vdash\Delta\varphi$.
\end{itemize}
\end{rem}
\begin{defn}
A \g logic has {\emph complete recursive axiomatization} whenever $\models\varphi$
if and only if $\vdash\varphi$ for every $\mathcal{L}$-sentence $\varphi$.
This kind of completeness is sometimes called {\emph weak completeness}.
\end{defn}
\begin{defn}
A \g logic is said to have the {\emph strong completeness} whenever for every $\mathcal{L}$-theory $T$
and $\mathcal{L}$-sentence $\varphi$, $T\models\varphi$
if and only if $T\vdash\varphi$.
\end{defn}
First-order \g logic $\mathfrak{G}_\mathbb{R}$ admits both kinds of completeness with respect to
any countable first-order language \cite{hajek98}. When $A\ne\emptyset$ example
\ref{fail strong completeness} shows that the strong completeness fails in
$\mathfrak{G}_{[0,1],A}$ while it is shown that $\mathfrak{G}_{[0,1],A}$ is
completely recursive axiomatizable \cite{esteva2009}.

One of the most useful tools in model theory of classical first-order logic
is the compactness theorem. In the case of mathematical fuzzy logic
this theorem has different aspects.
\begin{defn}
A \g logic is said to enjoy the
\emph{entailment compactness} whenever for any theory $T$ and sentence $\varphi$,
\begin{center}
$T\models\varphi$ if and only if $T\model{f}{}\varphi$.
\end{center}
\end{defn}
\begin{defn}
The \g logic $\mathfrak{G}_{V,A}$ ($\mathfrak{G}^\Delta_{V,A}$) has the \emph{approximate entailment compactness} property if for
every theory $T$ and sentence $\varphi$,
\begin{center}
$T\models\varphi$ if and only $T\models\bar{r}\to\varphi$ for all $r\in A\cup\{1\}$.
\end{center}
\end{defn}
\begin{defn}
We say that a \g logic has the \emph{compactness} property if for
every theory $T$,
\begin{center}
$T$ is satisfiable if and only if $T$ is finitely satisfiable.
\end{center}
\end{defn}
Since a proof is a finite sequence of conclusions, we have the following theorem.
\begin{thm}
If a logic admits the strong completeness then
it enjoys the triple kinds of compactness mentioned above.
\end{thm}
Specially, in the \g logic $\mathfrak{G}_\mathbb{R}$ both entailment compactness and compactness
hold.
\begin{thm} \cite{preining2003complete}
The entailment compactness and complete recursive axiomatization
(weak completeness) are equivalent in \g logic $\mathfrak{G}_V$.
\end{thm}
Furthermore, Prening \cite{preining2003complete} shows the \g logic
$\mathfrak{G}_V$ admits the entailment compactness property
if and only if either $V$ is a finite \g set or
the perfect kernel of $V$ includes $1$ or the perfect kernel of $V$ is
nonempty and $1$ is an isolated point of $V$. Particularly, he shows
that the entailment compactness fails in $\mathfrak{G}_V$ for countable \g set $V$.

Later, Pourmahdian et al. \cite{pourmahtavana2012} show that
if $V$ is a finite \g set or the perfect kernel of $V$ includes $1$ or $1$ is an isolated
point of $V$ then $\mathfrak{G}_V$ admits the compactness property.
\section{Compactness in \g logic $\mathfrak{G}_{[0,1],A}$}
In this section, we study the compactness property of \g logics $\mathfrak{G}_{[0,1],A}$ and also
$\mathfrak{G}^\Delta_{[0,1],A}$. From now on assume that $A'$ is denoted for the set of limit points of $A$ in a \g set $V$
with respect to the order topology on $V$. Firstly, note that if $A$ has a limit point $a\ne 0$ with respect to
the order topology and $a\in A\cup\{1\}$, then the compactness fails in $\mathfrak{G}_{[0,1],A}$ as well as
$\mathfrak{G}^\Delta_{[0,1],A}$.
\begin{exa}\label{limit}
Let $a\in A'\cap(A\cup\{1\})$ and assume that $\mathcal{L}=\{\rho\}$ where $\rho$ is a nullary predicate symbol.
Suppose that $\{r_i\}_{i=1}^\infty\subseteq A\setminus\{a\}$ is an increasing (decreasing)
sequence whose limit in $V$ is $a$. Let $T=\{\overline{a}\Rightarrow\rho\}\cup\{\rho\to\overline{r_i}\}_{i=1}^\infty$
($T=\{\rho\Rightarrow\overline{a}\}\cup\{\overline{r_i}\to\rho\}_{i=1}^\infty$).
Obviously $T$ is finitely satisfiable, but it is not satisfiable.
\end{exa}
\begin{exa}\label{limit11}
Let $a\in A'$ but $a\notin A\cup\{1\}$. Also assume that there is an increasing sequence
$\{r_i\}_{i=1}^\infty\subseteq A$ and a decreasing sequence $\{s_i\}_{i=1}^\infty\subseteq A$ so that
$\lim_i r_i=a=\lim_i s_i$. Let $\mathcal{L}=\{\rho,R(x)\}$ where $\rho$ is a nullary predicate symbol and
$R(x)$ is a unary predicate symbol. Let
$T=\{\exists x\,\big((\bar{r}_{i+1}\to R(x))\wedge(R(x)\to \bar{r}_i)\big)\}_{i=1}^\infty
\cup\{\bar{s}_i\Rightarrow\big(\forall x\,R(x)\big)\}_{i=1}^\infty\cup
\{\big(\forall x\,R(x)\big)\Rightarrow\rho\}\cup\{\rho\to\overline{r_i}\}_{i=1}^\infty$.
$T$ is finitely satisfiable, but it is not satisfiable. Indeed if $\mathcal{M}\models T$, then
$\big(\forall x\,R(x)\big)^\mathcal{M}=a$ and so the interpretation of $\rho$ in $\mathcal{M}$
makes no sense.
\end{exa}
Specially, $RGL$ does not admit the compactness property. However, if one consider some
non-standard truth value set, the compactness may hold on $RGL$. \cite{khatami2013rational}
prove that the compactness property is hold on $RGL$ within a semantic on the
non-standard truth value set $I=[0,1]^2\setminus\{(0,r): r>0\}$.

Now, using the Henkin construction, we show in the case that the set of limit
points of $A$ is at most $\{0\}$, the \g logic $\mathfrak{G}_{[0,1],A}$ admits the compactness
property. Observe that this method is based on
constructing the \g algebra of equivalence classes of
formulas modulo a theory, and then embedding this
\g algebra into the unit interval $[0,1]$, where
the {\emph countability} of the language $\mathcal{L}$
is a prerequisite necessary assumption for existence of
such an embedding. In the next section, we prove the
compactness property for some extensions of \g logics in which
the requirement of such an assumption is not obligatory.
\begin{defn}
The \g algebra with respect to the \g logic $\mathfrak{G}_{V,A}$
is a bounded lattice
$\mathbb{D}=\langle D, \dota, \dotv, 0^\mathbb{D}, 1^\mathbb{D}\rangle$ together with a
binary operation $\dotto$ and for each $r\in A\setminus\{0,1\}$ an element $r^\mathbb{D}\in D$ such that:
\begin{enumerate}
\item $\dota$ is the join (lub) operator and $\dotv$ is the meet (glb) operator.
\item $\dota$ and $\dotto$ form an adjoint pair, i.e., for all
  $a, b, c\in D$,
  \begin{center}
  $a\dota b\ge_\mathbb{D} c$ iff $a\ge_\mathbb{D} b\dotto c$,
  \end{center}
where $a\ge_\mathbb{D} b$ if and only if $a\dota b=a$.
\item $D$ is pre-linear, i.e, for all $a, b\in D$,
  \begin{center}
  $(a\dotto b)\dotv(b\dotto a)=0^\mathbb{D}$.
  \end{center}
\item $r^\mathbb{D}\dota s^\mathbb{D}=\max\{r,s\}^\mathbb{D}$.
\item $r^\mathbb{D}\dotto s^\mathbb{D}=0^\mathbb{D}$ iff $r\ge s$.
\item $r^\mathbb{D}\dotto s^\mathbb{D}=s^\mathbb{D}$ iff $r<s$.
\item $0^\mathbb{D}<r^\mathbb{D}<1^\mathbb{D}$ for all $0<r<1$.
\end{enumerate}
The \g algebra with respect to the \g logic $\mathfrak{G}^\Delta_{V,A}$ is formed
by the corresponding \g algebra with respect to $\mathfrak{G}_{V,A}$, i.e.,
$\mathbb{D}=\langle D, \dota, \dotv, \dotto, \{r^\mathbb{D}: r\in A\cup\{0,1\}\}\rangle$
together with a unary operation $\delta^\mathbb{D}$ which acts as follows.
\begin{itemize}
\item [(8)] $\delta^\mathbb{D}(0^\mathbb{D})=0^\mathbb{D}$.
\item [(9)] $\delta^\mathbb{D}(a)=1^\mathbb{D}$ for all $a\in\mathbb{D}\setminus\{0^\mathbb{D}\}$.
\end{itemize}
\end{defn}
\begin{exa}
The standard \g algebra with respect to $\mathfrak{G}_{[0,1],A}$ is
$[0,1]_A=\langle [0,1], \max, \min, \dotto, \{r : r\in A\cup\{0,1\}\}\rangle$. The
standard \g algebra corresponding to $\mathfrak{G}^*_n$ and $\mathfrak{G}^*_\downarrow$
is denoted by $[0,1]^*_n$ and $[0,1]^*_\downarrow$, respectively.
\end{exa}
\begin{lem}\label{embeddig godel algebra} If $A'\subseteq\{0\}$, then any countable
linear ordered \g algebra $\mathbb{D}$ with respect to $\mathfrak{G}_{[0,1],A}$ can be
continuously embedded into the standard \g algebra $[0,1]_A$ (i.e. an embedding that preserves
all suprema and infima that exist in $\mathbb{D}$). Particulary, any countable
linear ordered \g algebra $\mathbb{D}$ with respect to $\mathfrak{G}^*_n$ and
$\mathfrak{G}^*_\downarrow$ can be continuously embedded into the standard \g algebra
$[0,1]^*_n$ and $[0,1]^*_\downarrow$, respectively.
\end{lem}
\begin{proof}
As \cite[Lemma 5.3.1]{hajek98}, set $D'= D\times\{0\}\cup\bigcup
\{\{u\}\times((0,1)\cap\mathbb{Q}): u\text{ has no successor in D}\}$
which ordered lexicographically by induced ordering $\le_\mathbb{D}$ of $\mathbb{D}$.
By setting $r^{\mathbb{D}'}=(r,0)$ for each nullary connective $\bar{r}$, one can easily construct a countable densely
linearly ordered \g algebra $\mathbb{D}'$. Furthermore, the mapping $u\to(u,0)$ is a continuous embedding
from $\mathbb{D}$ into $\mathbb{D}'$ wherein the image of $r^\mathbb{D}$ is $r^{\mathbb{D}'}$.
There are two cases.
\begin{itemize}
\item [Case 1:] $A'=\emptyset$. Thus, $A$ is a finite set. So, the proof is similar to
the proof for \g logic \cite[Lemma 5.3.2]{hajek98} with an easy
adaptation of the back and forth method for embedding countable densely linearly
ordered \g algebra $\mathbb{D}'$ into $[0,1]_A$.
\item [Case 2:] $A'=\{0\}$. So, there is a decreasing sequence $\{r_i\}_{i\in\mathbb{N}}$ in the
open unit interval $(0,1)$ so that $A=\{r_i\}_{i\in\mathbb{N}}$ and $\lim_i r_i=0$.
Let $\|u\|=\inf\{r: u\le_{_{\mathbb{D}'}} r^{\mathbb{D}'}\}$
for any $u\in D'$ .
Define the equivalence relation $\sim$ on $D'$ by
\begin{center}
$u\sim v$ if and only if $\|u\|=\|v\|$.
\end{center}
Now, we have
\begin{itemize}
\item $[0^{\mathbb{D'}}]_\sim=\{u\in D': u\le_{_{\mathbb{D'}}}r_i^{\mathbb{D}'} \text{ for all } i\in\mathbb{N}\}$,
\item if $\|u\|=r_i$ then $[u]_\sim=\{u\in D': r_{i+1}^{\mathbb{D}'}<_{_{\mathbb{D}'}}u
\le_{_{\mathbb{D'}}}r_i^{\mathbb{D}'}\}$,
\item $[1^{\mathbb{D'}}]_\sim=\{u\in D': r_1^{\mathbb{D}'}<_{_{\mathbb{D}'}}u
\le_{_{\mathbb{D'}}}1^{\mathbb{D}'}\}$.
\end{itemize}
For each $u\in D'$ if $\|u\|=r\in A$, obviously $[u]_\sim$ can be continuously embedded
into $(r_{i+1},r_i]$ by means of a function $f_r$. Also $[1^{\mathbb{D'}}]_\sim$ continuously embedded
into $(r_1,1]$ by a function like as $f_1$. Let $f_0$ be the trivial constant function from $[0^{\mathbb{D'}}]_\sim$ into $\{0\}$.
Now, the function $f=\cup\{f_r: r\in A\cup\{0,1\}\}$ fulfills the proof.
\end{itemize}
\end{proof}
\subsection{Usual Compactness}\hfill

As already mentioned, if $A\ne\emptyset$ then for any $0<r<1$, $T=\{\bar{r}\}$ is a consistent theory
in $\mathfrak{G}_{V,A}$ which is not satisfiable. So, when $A\ne\emptyset$ we use the
"strongly consistency" instead of "consistency".
\begin{defn} An $\mathcal{L}$-theory $T$ is called strongly consistent
if $T\nvdash\bar{r}$ for $r\in A\cup\{1\}$ (i.e, $r>0$).
\end{defn}
Observe that every satisfiable theory is strongly consistent.
We show that when $A'\subseteq\{0\}$, strongly consistent theories are satisfiable.
Note that Examples \ref{limit} and \ref{limit11} gives strongly consistent theories which are not satisfiable in
the \g logic $\mathfrak{G}_{[0,1],A}$ when $A'\nsubseteq\{0\}$.

Two different concepts "finitely entailment" and "proof" bring us two kinds
of Henkin and complete theories.
\begin{defn} Let $T$ be an $\mathcal{L}$-theory.
\begin{enumerate}
\item $T$ is Henkin if for every universal $\mathcal{L}$-formula $\forall x\ \varphi(x)$
that is not finitely entailed by $T$, there is a witness constant symbol $c$ in $\mathcal{L}$ such that
$T\nmodel{f}{}\varphi(c)$.
\item $T$ is deductively Henkin or d-Henkin if for every universal $\mathcal{L}$-formula $\forall x\ \varphi(x)$
that is not proved by $T$, there is a witness constant symbol $c$ in $\mathcal{L}$ such that
$T\nvdash\varphi(c)$.
\item $T$ is called a complete theory if for any pair of $\mathcal{L}$-sentences $(\varphi, \psi)$,
either $T\model{f}{}\varphi\to\psi$ or $T\model{f}{}\psi\to\varphi$.
\item $T$ is called deductively complete or d-complete theory if for any pair of
$\mathcal{L}$-sentences $(\varphi, \psi)$, either $T\vdash\varphi\to\psi$ or $T\vdash\psi\to\varphi$.
\end{enumerate}
\end{defn}
The following theorem leads to deduced the compactness property for the
\g logic $\mathfrak{G}_{[0,1],A}$ when $A'\subseteq\{0\}$.
\begin{thm}\label{compactness of dmaximal theories gdowwn}
Let $\mathcal{L}$ be a countable first-order language.
If $A'\subseteq\{0\}$ then every strongly consistent d-complete d-Henkin $\mathcal{L}$-theory in
$\mathfrak{G}_{[0,1],A}$ is satisfiable.
\end{thm}
\begin{proof}
Let $T$ be a strongly consistent deductively complete d-Henkin $\mathcal{L}$-theory.
Also let $Lind(T)$ be the class of all $T$-provably equivalent $\mathcal{L}$-sentences, i.e., the equivalence
classes $[\varphi]_T$ of all $\mathcal{L}$-sentences $\varphi$ modulo to the following equivalence relation.
\begin{center}
$\varphi\sim\psi$ if and only if $T\vdash\varphi\leftrightarrow\psi$.
\end{center}
Define an ordering $\lesssim$ on $Lind(T)$ as follows
\begin{center}
$[\varphi]_T\lesssim[\psi]_T$ if and only if $T\vdash\psi\to\varphi$.
\end{center}
Because $T$ is a complete theory, $(Lind(T),\lesssim)$ is a linearly ordered set.
Now, we obtain a countable linearly ordered \g algebra $\mathbb{L}_T$ from $Lind(T)$ by setting,
\begin{center}
$[\varphi]_T\dotv[\psi]_T=[\varphi\vee\psi]_T$,\\
$[\varphi]_T\dota[\psi]_T=[\varphi\wedge\psi]_T$,\\
$[\varphi]_T\dotto[\psi]_T=[\varphi\to\psi]_T$,\\
$r^\mathbb{D}=[\bar{r}]_T$ for any nullary connective $\bar{r}$.
\end{center}
Axiom $RG3$ together with the strongly consistency of $T$ implies that
$[\bar{r}]_T\lnsim[\bar{s}]_T$ for each $r<s$ in $A$. Thus, by Lemma \ref{embeddig godel algebra}
there is an embedding $g$ from $\mathbb{L}_T$ into the standard \g algebra $[0,1]_A$
such that $[\bar{r}]_T$ mapped to $r$.

The canonical $\mathcal{L}$-structure $\mathcal{M}_T$ of $T$ is made as follows.
\begin{itemize}
\item[a)] The universe of $\mathcal{M}_T$ is the set of all closed $\mathcal{L}$-terms $CM(T)$.
\item[b)] For each n-ary function symbol $f$, define $f^{\mathcal{M}_T}:CM(T)^n\to CM(T)$
by
\begin{center}
$f^{\mathcal{M}_T}(t_1, ..., t_n)=f(t_1, ..., t_n)$.
\end{center}
\item[c)] For each n-ary predicate symbol $P$, define $P^{\mathcal{M}_T}:CM(T)^n\to[0,1]$
by
\begin{center}
$P^{\mathcal{M}_T}(t_1, ..., t_n)=g([P(t_1, ..., t_n)]_T)$.
\end{center}
\end{itemize}
\end{proof}
For the case that $A'=\emptyset$, the compactness property of $\mathfrak{G}_{[0,1],A}$ follows from
Theorem \ref{compactness of dmaximal theories gdowwn} and the following theorem.
\begin{thm}\label{compactness gn}
Let $A'=\emptyset$. Every strongly consistent
$\mathcal{L}$-theory in $\mathfrak{G}_{[0,1],A}$ is contained in
a strongly consistent d-complete d-Henkin $\mathcal{L}'$-theory
such that $\mathcal{L}\subseteq\mathcal{L}'$.
\end{thm}
\begin{proof}
Let $T$ be a strongly consistent $\mathcal{L}$ theory and
$\mathcal{L}'$ be the extended of $\mathcal{L}$ with countably many new constant symbols.
Enumerate all pairs of $\mathcal{L}'$-sentences by $\{(\theta_i,\psi_i)\}_{i\in\mathbb{N}}$. Also assume
that $\{\varphi_i(x)\}_{i\in\mathbb{N}}$ be the set of all $\mathcal{L}'$-formulas with one free variable.
Now, we construct inductively sequences $\{T_n\}_{n\in\mathbb{N}}$ of $\mathcal{L}'$-theories and $\{\chi_{_n}\}_{n\in\mathbb{N}}$
of $\mathcal{L}'$-sentences such that for each $n\in\mathbb{N}$, $T_n\nvdash\chi_{_n}$.\\
\underline{stage 0:} Let $T_0=T$ and
\begin{center}
$\chi_{_0}=\left\{\begin{array}{cc}
1&A=\emptyset,\\
\min_{r\in A}\{r\}&otherwise.
\end{array}\right.$
\end{center}
Obviously, $\chi_{_0}>0$ and since $T$ is strongly consistent, $T_0\nvdash\chi_{_0}$.\\
\underline{stage n+1=2i:} Let $\chi_{_{n+1}}=\chi_{_n}$. Now, if $T_n\cup\{\theta_i\to\psi_i\}\nvdash\chi_{_n}$
set $T_{n+1}=T_n\cup\{\theta_i\to\psi_i\}$ and otherwise set $T_{n+1}=T_n\cup\{\psi_i\to\theta_i\}$. Since
$T_n\nvdash\chi_{_n}$, either $T_n\cup\{\theta_i\to\psi_i\}\nvdash\chi_{_n}$ or
$T_n\cup\{\psi_i\to\theta_i\}\nvdash\chi_{_n}$. Thus $T_{n+1}\nvdash\chi_{_{n+1}}$.\\
\underline{stage n+1=2i+1:} Let $c_i$ be a constant symbol of $\mathcal{L}'$ not occurring in $\varphi_i(x)$ and
the constructed objects until the current stage. Consider two cases.
\begin{quote}
\begin{itemize}
\item[Case 1:] If $T_n\nvdash\chi_{_n}\vee\varphi_i(c_i)$ let
$T_{n+1}=T_n$ and $\chi_{_{n+1}}=\chi_{_n}\vee\varphi_i(c_i)$.
Since $T_n\nvdash\chi_{_n}$, clearly in this case $T_{n+1}\nvdash\chi_{_{n+1}}$.
\item[Case 2:] If $T_n\vdash\chi_{_n}\vee\varphi_i(c_i)$ set
$T_{n+1}=T_n\cup\{\chi_{_n}\to\forall x\,\varphi_i(x)\}$ and $\chi_{_{n+1}}=\chi_{_n}$.
Since $T_n\vdash\chi_{_n}\vee\varphi_i(c_i)$ using (Gen) and (G$\forall$3) we have
$T_n\vdash\chi_{_n}\vee\forall x\,\varphi_i(x)$. So, by definition of the connective $\vee$ and the fact that
$T_n\nvdash\chi_{_n}$ we have $T_n\cup\{\forall x\,\varphi_i(x)\to\chi_{_n}\}\vdash\chi_{_n}$.
Thus, using the proof-by-case property and the fact that $T_n\nvdash\chi_{_n}$,
we have $T_n\cup\{\chi_{_n}\to\forall x\,\varphi_i(x)\}\nvdash\chi_{_n}$ that is $T_{n+1}\nvdash\chi_{_{n+1}}$.
\end{itemize}
\end{quote}
Now, let $T'=\cup_{n\in\mathbb{N}}T_n$. Clearly $T'$ is strongly consistent, since otherwise
if $T'\vdash\bar{r}$ for some $r\in A\cup\{1\}$ then by (RG2(a)) $T'\vdash\chi_{_0}$. So, for
some $n\in\mathbb{N}$, $T_n\vdash\chi_{_0}$ which implies that $T_n\vdash\chi_{_n}$, a contradiction.

On the other hand, clearly $T'$ is deductively complete. Now, if $T'\nvdash\forall x\,\varphi_i(x)$
then $T_{2i+1}\nvdash\chi_{_{2i+1}}\vee\varphi_i(c_i)$, since otherwise by case 2 of stage n+1 we have
$T_{2i+1}\vdash\chi_{_{2i+1}}\vee\forall x\,\varphi_i(x)$ which implies that
$T_{2i+2}\vdash\forall x\,\varphi_i(x)$, a contradiction. Thus, $T_{2i+2}=T_{2i+1}$ and
$\chi_{_{2i+2}}=\chi_{_{2i+1}}\vee\varphi_i(c_i)$. But then $T'\nvdash\varphi_i(c_i)$, since
otherwise $T'\vdash\chi_{_{2i+2}}$, a contradiction. So, $T'$ is a deductively complete
d-Henkin $\mathcal{L}'$-theory.
\end{proof}
\begin{cor}\label{final compactness gn}
For countable first-order language $\mathcal{L}$, the \g logic $\mathfrak{G}^*_n$ admit the compactness property.
\end{cor}
\begin{rem}
By Example \ref{fail strong completeness} we know that the strong completeness fails in $\mathfrak{G}^*_n$.
Indeed, when $T\cup\{\varphi\}\subseteq Sent(\mathcal{L})$, $A\ne\emptyset$, and $T\nvdash\varphi$ one could
not obtain a deductively complete d-Henkin  extension $T'$ of $T$ such that $T'\nvdash\varphi$ in
the \g logic $\mathfrak{G}_{[0,1],A}$.
For example, the theory $T=\{\neg\neg\rho\to\bar{r}\}$ in Example \ref{fail strong completeness} could not be extend
to a deductively complete theory $T'$ such that $T'\nvdash\neg\rho$.
\end{rem}
The method used in Theorem \ref{compactness gn} could not be used for the
case that $A'=\{0\}$. To prove the compactness property of $\mathfrak{G}_{[0,1],A}$ for
the case that $A'=\{0\}$ we use the following lemma.
\begin{lem}\label{maximally finitely satisfiable}
Let $T$ be a maximally strongly consistent $\mathcal{L}$-theory and $\varphi$ and $\psi$ be
two arbitrary $\mathcal{L}$-sentences. For the \g logic $\mathfrak{G}_{[0,1],A}$ we have,
\begin{enumerate}
\item $T$ is deductively complete,
\item if $\varphi\vee\psi\in T$, then either $\varphi\in T$ or $\psi\in T$,
\item if $A'=\{0\}$ and $\bar{r}\to\varphi\in T$ for $r\in A\cup\{1\}$, then $\varphi\in T$.
\end{enumerate}
\end{lem}
\begin{proof}
(1) and (2) are straightforward. For (3) we show that $T\cup\{\varphi\}$ is
strongly consistent. Suppose, , to derive a contradiction, that $T\cup\{\varphi\}$ is not strongly consistent.
So, there is $r\in A\cup\{1\}$ such that $T\cup\{\varphi\}\vdash\bar{r}$. Thus, $T\vdash\varphi\to\bar{r}$.
Since $A'=\{0\}$, there is $s\in A\cup\{1\}$ such that $s<r$. By the assumption $\bar{s}\to\varphi\in T$. i.e, $T\vdash\bar{s}\to\varphi$.
Hence, by transitivity property of proof $T\vdash\bar{s}\to\bar{r}$ and by $RG2(b)$, $T\vdash\bar{r}$. A contradiction.
\end{proof}
Observe that by Zorn's lemma, any strongly consistent $\mathcal{L}$-theory $T$ contained in a
maximally strongly consistent $\mathcal{L}$-theory. The following theorem show that this maximally strongly
consistent extension could be chosen in a language $\mathcal{L}'\supseteq\mathcal{L}$ such that it is
Henkin. So, in the light of Theorem \ref{compactness of dmaximal theories gdowwn} the compactness
property of $\mathfrak{G}_{[0,1],A}$ is established for the case that $A'=\{0\}$ and $\mathcal{L}$ is
a countable first-order language.
\begin{thm}\label{comp1} Let $A'=\{0\}$. Every strongly consistent
$\mathcal{L}$-theory in $\mathfrak{G}_{[0,1],A}$ is contained in
a maximally strongly consistent deductively Henkin $\mathcal{L}'$-theory
such that $\mathcal{L}\subseteq\mathcal{L}'$.
\end{thm}
\begin{proof}
Let $T$ be a strongly consistent $\mathcal{L}$-theory. $T'$ will be constructed in countably
many phases. Indeed, $T'$ is a maximally strongly consistent theory
containing the union of countably many maximally strongly
consistent $\mathcal{L}_i$-theories $T_i$ in which for every
$i\ge 1$, $\mathcal{L}_i$ have a witness constant for each unprovable sentence
$\forall x\ \varphi(x)$ where $\varphi(x)\in Form(\mathcal{L}_{i-1})$. To this end,
consider the following notions.
\begin{itemize}
\item $\mathcal{L}_0=\mathcal{L}$.
\item $F_0=Form(\mathcal{L}_0)$ and for $i\ge 1$, $F_i=Form(\mathcal{L}_i)\setminus Form(\mathcal{L}_{i-1})$.
\item For each $i\ge 1$, $\mathcal{L}_i=\mathcal{L}_{i-1}\cup\{c_{\varphi(x),r,s}:\varphi(x)\in F_{i-1},
r,s\in A\cup\{1\}, r>s\}$ where each $c_{\varphi(x),r,s}$ is a new constant symbol.
\item $T'_0=T$.
\item For nullary connectives $\bar{r}$ and $\bar{s}$ and formula $\varphi(x)$,
$\theta_{\varphi(x),r,s}=(\bar{r}\to\forall x\,\varphi(x))\vee(\varphi(c_{\varphi(x),r,s})\to\bar{s})$.
\item For each $i\ge 1$, $T'_i=T_{i-1}\cup\{\theta_{\varphi(x),r,s}:\varphi(x)\in F_{i-1},
r,s\in A\cup\{1\}, r>s\}$ where $T_{i-1}$ is a maximally strongly
consistent $\mathcal{L}_{i-1}$-theory containing $T'_{i-1}$.
\end{itemize}

Firstly, we show that for each $i\ge 0$, $T'_i$ is strongly consistent. Obviously, $T'_0$ is strongly consistent.
Assume that for each $k<n$, $T'_k$ is strongly consistent. Suppose that, on the contrary, $T'_n$ is not
strongly consistent. Thus, there exists $t\in A\cup\{1\}$ such that $T'_n\vdash\bar{t}$. Hence, there is a finite subset
$S$ of $T_{n-1}$ such that $S\cup\{\theta_{\varphi_i(x),r_i,s_i}\}_{i=1}^m\vdash\bar{t}$ and no proper subset of
$S\cup\{\theta_{\varphi_i(x),r_i,s_i}\}_{i=1}^m$ proves $\bar{t}$. Set,
$\Gamma=\{\theta_{\varphi_i(x),r_i,s_i}\}_{i=1}^{m-1}$. By deduction theorem,
$S\cup\Gamma\vdash\theta_{\varphi_m(x),r_m,s_m}\to\bar{t}$. Consider the abbreviations
$\theta_m$ and $c_m$ for $\theta_{\varphi_m(x),r_m,s_m}$ and $c_{\varphi_m(x),r_m,s_m}$, respectively.
Since $\varphi(c_m)\to\bar{s}_m\vdash\theta_m$ we have
$S\cup\Gamma\vdash(\varphi(c_m)\to\bar{s}_m)\to\bar{t}$, which leads to
deduce that $T_{n-1}\cup\Gamma\vdash\bar{s}_m\to\forall x\,\varphi_m(x)$.
On the other hand, as $\bar{r}_m\to\forall x\,\varphi_m(x)\vdash\theta_m$ we
have $S\cup\Gamma\vdash(\bar{r}_m\to\forall x\,\varphi_m(x))\to\bar{t}$ and
so one could conclude that $T_{n-1}\cup\Gamma\vdash\forall x\,\varphi_m(x)\to\bar{r}_m$.
Hence, $T_{n-1}\cup\Gamma\vdash\bar{s}_m\to\bar{r}_m$ and so
$T_{n-1}\cup\Gamma\vdash\bar{t}$ which is a contradiction.

Secondly, let $\mathcal{L}'=\bigcup_{n\ge 0}\mathcal{L}_n$ and take a
maximally strongly consistent $\mathcal{L}'$-theory $T'$, containing $\bigcup_{n\ge 0} T_n$. $T'$ is
provably Henkin. Verily, if $T'\nvdash\forall x\,\varphi(x)$ for some $\varphi(x)\in
Form(\mathcal{L}')$ then by maximality of $T'$ and Lemma \ref{maximally finitely satisfiable}-3
there is $r\in A\cup\{1\}$ such that $\bar{r}\to\forall x\,\varphi(x)\notin T'$.
Now, as $A'=\{0\}$ take $s\in A\cup\{1\}$ such that $s<r$. As, $(\bar{r}\to\forall x\,\varphi(x))\vee(\varphi(c_{\varphi(x),r,s})\to\bar{s})\in T'$, maximality of $T'$
and Lemma \ref{maximally finitely satisfiable}-2 implies that
$\varphi(c_{\varphi(x),r,s})\to\bar{s}\in T'$. Thus, by Lemma \ref{maximally finitely satisfiable}-3
$T'\nvdash\varphi(c_{\varphi(x),r,s})$, and the proof is completed.
\end{proof}
\begin{cor}\label{compactness gn gdown}
Let $\mathcal{L}$ be a countable first-order language.
If $A'=\{0\}$ then $\mathfrak{G}_{[0,1],A}$ satisfy the compactness property.
Specially, $\mathfrak{G}^*_\downarrow$ admits the compactness property.
\end{cor}
\subsection{Entailment Compactness}\hfill

Now, we study the entailment compactness and approximate entailment compactness in
\g logics $\mathfrak{G}_{[0,1],A}$ and $\mathfrak{G}^\Delta_{[0,1],A}$. Note that
the usual compactness follows from the entailment compactness. However, the method
we use in this subsection based on the notion of "finitely entailment" while
the method used in the previous subsection is based on the notion of "proof" and the
concept of "strongly consistency".

The following example show that the entailment compactness fails on $\mathfrak{G}^*_\downarrow$.
\begin{exa}\label{fail entail compact gdown}
let $\mathcal{L}=\{\rho\}$ where $\rho$ is a nullary predicate symbol.
Let $T=\{\overline{\frac{1}{n}}\to\rho\}_{n\in\mathbb{N}}$. One can easily verify that
in the \g logic $\mathfrak{G}^*_\downarrow$, $T\models\rho$ but $T\nmodel{f}{}\rho$.
\end{exa}
However, when $A'=\{0\}$, the approximate entailment compactness holds in $\mathfrak{G}_{[0,1],A}$.
\begin{thm}\label{approx entail com theorem}
Let $\mathcal{L}$ be a countable first-order language,
$T$ be an $\mathcal{L}$-theory, and $\varphi$ be an $\mathcal{L}$-sentence. If $A'=\{0\}$ then
the \g logic $\mathfrak{G}_{[0,1],A}$ enjoys the approximate entailment compactness.
Particularly, in $\mathfrak{G}^*_\downarrow$ we have,
\begin{center}
$T\models\varphi$ if and only if $T\model{f}{}\overline{\frac{1}{n}}\to\varphi$
for all integers $n\ge 1$.
\end{center}
\end{thm}
\begin{proof}
Let $T\models\varphi$. We want to show that $T\model{f}{}\bar{r}\to\varphi$
for all $r\in A\cup\{1\}$. Suppose not. So, there exists $r_0\in A\cup\{1\}$
such that for any finite subset $S$ of $T$,
$S\nmodels\overline{r_0}\to\varphi$. Thus, for any finite subset
$S$ of $T$ there is a model $\mathcal{M}$ of $S$
such that $\mathcal{M}\nmodels\overline{r_0}\to\varphi$, which means that
$\mathcal{M}\models\varphi\to\overline{r_0}$. Thus, for any finite subset
$S$ of $T$, $S\cup\{\varphi\to\overline{r_0}\}$ is satisfiable. Hence,
compactness property of $\mathfrak{G}_{[0,1],A}$ implies that $T\cup\{\varphi\to\overline{r_0}\}$ is satisfiable which is
in contradict with $T\models\varphi$.\\
Conversely, suppose that for any $r\in A\cup\{1\}$, $T\model{f}{}\bar{r}\to\varphi$.
We want to show that $T\models\varphi$. Since $T$ has a model, by
\emph{reductio ad absurdum} suppose that there exists a model
$\mathcal{M}$ of $T$ such that $\mathcal{M}\nmodels\varphi$. But because
$A'=\{0\}$, one could find $r\in A\cup\{1\}$ such that $\varphi^\mathcal{M}\ge r$. This means for any
finite subset $S$ of $T$, $S\nmodels\bar{r}\to\varphi$, a contradiction.
\end{proof}
On the other hand, $\mathfrak{G}^*_n$ enjoys the entailment compactness. This follows from the
following two lemmas.
\begin{lem}\label{compactness of maximal theories}
Let $\mathcal{L}$ be a countable first-order language. Let $A'=\emptyset$,
$T$ be a complete Henkin $\mathcal{L}$-theory, and $\varphi$ be an
$\mathcal{L}$-sentence. In the \g logic $\mathfrak{G}_{[0,1],A}$,
$T\models\varphi$ if and only if $T\model{f}{}\varphi$.
\end{lem}
\begin{proof}
Let $Lind(T)$ be the class of all $T$-equivalent $\mathcal{L}$-sentences, i.e., the equivalence
classes $[\varphi]_T$ of all $\mathcal{L}$-sentences $\varphi$ modulo to the following equivalence relation.
\begin{center}
$\varphi\sim\psi$ if and only if $T\model{f}{}\varphi\leftrightarrow\psi$.
\end{center}
By the same way as the proof of Theorem \ref{compactness of dmaximal theories gdowwn} and
replacing $\vdash$ by $\model{f}{}$ we obtain an $\mathcal{L}$-structure $\mathcal{M}_T\models T$. Now, let $T\models\varphi$ but
$T\nmodel{f}{}\varphi$. So, $[\varphi]_T\gnsim[\bar{0}]_T$. But then since $A'=\emptyset$,
the proof of Lemma \ref{embeddig godel algebra} show that $\varphi^\mathcal{M}=g([\varphi]_T)>0$, a
contradiction.
\end{proof}
\begin{rem}
If $A'=\{0\}$ and $[\varphi]_T\gnsim[\bar{0}]_T$, then $g([\varphi]_T)$ does not necessarily grater than $0$ (cf. proof of Lemma \ref{embeddig godel algebra}, case 2).
\end{rem}
\begin{lem}\label{co1}
Let $T$ be an $\mathcal{L}$-theory, $\varphi$ be an $\mathcal{L}$-sentence, and $T\nmodel{f}{}\varphi$.
The followings are hold in \g logics $\mathfrak{G}^\Delta_{[0,1],A}$.
\begin{enumerate}
\item There exists a complete $\mathcal{L}$-theory $\overline{T}\supseteq T$ such that $\overline{T}\nmodel{f}{}\varphi$.
\item There exists a first-order language $\mathcal{L}'\supseteq\mathcal{L}$ and a complete Henkin $\mathcal{L}'$-theory $T'\supseteq T$ such that $T'\nmodel{f}{}\varphi$.
\end{enumerate}
\end{lem}
\begin{proof}
Proof of (1) is straightforward. For (2) assume that
\begin{itemize}
\item $\mathcal{L}_0=\mathcal{L}$, $T_0=T$, $\epsilon_0=\varphi$,
\item for $n\ge 0$, $\overline{T}_n$ be a complete theory containing $T_n$ such that
$\overline{T}_n\nmodel{f}{}\epsilon_n$,
\item for $n\ge 0$, $\mathcal{L}_{n+1}=\mathcal{L}_n\cup\{\epsilon_{n+1}\}
\cup\{c_\psi: \overline{T}_n\nmodel{f}{}\forall x\,\psi(x)\}$,
where $\epsilon_{n+1}$ is a new nullary predicate symbol and
each $c_\psi$ is a new constant symbol,
\item for $n\ge 0$, $T_{n+1}=\overline{T}_n\cup\{\epsilon_{n}\to\epsilon_{n+1}\}
\cup\{\psi(c_\psi)\to\epsilon_{n+1}: \overline{T}_n\nmodel{f}{}\forall x\,\psi(x)\}$.
\end{itemize}
We show that for each $n\ge 0$, $T_n\nmodel{f}{}\epsilon_n$. Obviously
$T_0\nmodel{f}{}\epsilon_0$. Assume that $T_n\nmodel{f}{}\epsilon_n$.
We show that $T_{n+1}\nmodel{f}{}\epsilon_{n+1}$. To this
end, let $A=B\cup\{\psi_i(c_{\psi_i})\to\epsilon_{n+1}:
\overline{T}_n\nmodel{f}{}\forall x\,\psi_i(x)\}_{i=1}^m$ be a finite subset of $T_{n+1}$
in which $B$ is a finite subset of $\overline{T}_n$. Since
$\overline{T}_n\nmodel{f}{}\epsilon_n$ and for each $1\le i\le m$,
$\overline{T}_n\nmodel{f}{}\forall x\,\psi_i(x)$ and $\overline{T}_n$
is a complete $\mathcal{L}_n$-theory,
$\overline{T}_n\nmodel{f}{}\epsilon_n\vee\big(\ds\bigvee_{1\le i\le m}\forall x\,\psi_i(x)\big)$.
Now, as $B$ is a finite subset of $\overline{T}_n$, there is an
$\mathcal{L}_n$-structure $\mathcal{M}\models B$ such that
$\min\{\epsilon^\mathcal{M}_n, \big(\forall x\,\psi_1(x)\big)^\mathcal{M}, ...,
\big(\forall x\,\psi_m(x)\big)^\mathcal{M}\}=\alpha>0$. Interpreting
$\epsilon_{n+1}$ in $\mathcal{M}$ by a nonzero rational number less that $\alpha$
leads to the fact that $\mathcal{M}\models A$ and $\mathcal{M}\nmodels\epsilon_{n+1}$.

Now, let $\mathcal{L}'=\cup_{n=0}^\infty \mathcal{L}_n$ and
$T^*=\cup_{n=0}^\infty T_n$. One could easily verify that $T^*\nmodel{f}{}\varphi$.
Let $T'$ be a complete $\mathcal{L}'$-theory containing $T^*$ such that $T'\nmodel{f}{}\varphi$.
Obviously the construction implies that $T'$ is a complete Henkin $\mathcal{L}'$-theory.
\end{proof}
Now, in the light of Lemma \ref{compactness of maximal theories} and
Lemma \ref{co1} we deduced the entailment compactness of $\mathfrak{G}^*_n$.
\begin{cor}\label{final entailment compactness gn}
Let $\mathcal{L}$ be a countable first-order language. $\mathfrak{G}^*_n$ enjoys the
entailment compactness property.
\end{cor}
\begin{rem}\label{delta2}
Let $\mathcal{L}$ be first-order language.When $A'=\emptyset$ one could easily modify the proof of
Lemma \ref{embeddig godel algebra} and Lemma \ref{compactness of maximal theories}
to the case of \g logics equipped with the unary connective $\Delta$. So by Lemma \ref{co1},
$\mathfrak{G}^\Delta_{[0,1],A}$ admit the compactness as well as the entailment compactness.
\end{rem}
Despite the compactness property of $\mathfrak{G}^\Delta_{[0,1],A}$ for finite
set $A$, this fails in $\mathfrak{G}^\Delta_{[0,1],A}$ if $A'=\{0\}$.
\begin{exa}\label{gdeltaconter1}
Let $\mathcal{L}$ contain a nullary predicate symbol $\rho$ and let
$T=\{\rho\to\overline{\frac{1}{n}}\}_{n\in\mathbb{N}}\cup\{\neg(\Delta(\rho))\}$.
One can easily verify that in the \g logic
$\mathfrak{G}^\Delta_{[0,1],(0,1)_\downarrow}$, $T$ is finitely satisfiable but it is not satisfiable.
\end{exa}
\section{Compactness when the underlying language is uncountable}
As already mentioned, all the compactness results in \g logics are restricted
by the countability of the underlying language. The following example shows that when the
underlying language is uncountable, the compactness fails in almost all extensions of \g logics.
\begin{exa}\cite{mathover}
Let $\mathcal{L}$ be a relational language contains uncountably many unary predicate
symbols $\{R(x)\}\cup\{\rho_i(x)\}_{i\in(\omega_1+1)}$. Set,
\begin{center}
$T=\{\neg\forall x\,R(x), \forall x\,\big(\rho_1(x)\Rightarrow R(x)\big)\}
\cup\{\forall x\,\big(\rho_i(x)\Rightarrow \rho_j(x)\big): i>j\}_{i,j\in(\omega_1+1)}$
\end{center}
Clearly, in any \g logic $\mathfrak{G}_{[0,1],A}$, $T$ is finitely satisfiable but it is not satisfiable.
\end{exa}
However, when $V'\subseteq\{0\}$ we show that the \g logic $\mathfrak{G}_{V,A}$ may admit the compactness property,
even for uncountable first order languages. Note that in the above example, $T$ is not finitely
satisfiable in \g logics with truth value set $V$ such that $V'\subseteq\{0\}$.

We prove the compactness theorem by the ultraproduct method. For doing this we need a
similarity relation. In metrically semantic of t-norm based fuzzy logics,
the interpretation of similarity relation is a pseudo-metric.
\subsection{Ultrametric structure}\hfill

One of the most important tools in classical first-order logic
is the equality relation. The advantage of this relation appears almost
in all results of model theory of classical first-order logic.

Equality relation has some common properties, called \emph{Similarity Axioms}.
\begin{itemize}
\item[(1)] (Reflexivity) $\forall x\,(x\approx x)$.
\item[(2)] (Symmetry) $\forall x\forall y\,(x\approx y\to y\approx x)$.
\item[(3)] (Transitivity) $\forall x\forall y \forall z\,\left((x\approx y\wedge y\approx z)\to x\approx z\right)$.
\end{itemize}
Let $d$ be a binary predicate symbol in the language $\mathcal{L}_d$
that we want to role as equality relation in classical first-order logic.
If $\mathcal{M}$ be an $\mathcal{L}_d$-structure such that
\begin{center}
$\mathcal{M}\models\left\{\forall x\,d(x,x),
\forall x\forall y\,\left(d(x,y)\to d(y,x)\right),
\forall x\forall y \forall z\,\left((d(x,y)\wedge d(y,z))\to d(x,z)\right)\right\}$,
\end{center}
then it is easy to see that $d^\mathcal{M}$ must be a pseudo-ultrametric on the
universe of $\mathcal{M}$ (a pseudo-metric that
satisfies the strong triangle inequality, that is $d^\mathcal{M}(a,b)\le\max\{d^\mathcal{M}(a,c),d^\mathcal{M}(b,c)\}$ for all $a,b,c\in M$).
\begin{defn}
For a given language $\mathcal{L}_d$, containing a binary predicate symbol $d$, an
$\mathcal{L}_d$-\emph{ultrametric structure} or simply
an \emph{ultrametric structure} is an $\mathcal{L}_d$-structure $\mathcal{M}$ where
$(M,d^\mathcal{M})$ is an ultrametric space.
\end{defn}
\begin{exa}
Any first-order structure with the discrete metric is an ultrametric structure.
\end{exa}
\begin{exa}
Let $(M,d)$ be an ultrametric space. Define $d'(x,y)=\displaystyle\frac{d(x,y)}{1+d(x,y)}$.
$(M,d')$ is an ultrametric structure.
\end{exa}
In classical first-order logic when $P$ is an n-ary predicate symbol and $f$ is an
n-ary function symbol and $\bar{a}\in M^n$ is "equal" to $\bar{b}\in M^n$,
we have $P^\mathcal{M}(\bar{a})=P^\mathcal{M}(\bar{b})$ (as subsets of
$M^n$) and also $f^\mathcal{M}(\bar{a})$ is "equal" to
$f^\mathcal{M}(\bar{b})$. These properties express the extensional identically of
$\bar{a}$ and $\bar{b}$ with respect to the equality relation, which is called \emph{Congruence Axioms}.
\begin{itemize}
\item[(4)]For each n-ary predicate symbol $P$,\\
$\forall x_1...\forall x_n\forall y_1...\forall y_n\big(\wedge_{i=1}^n(x_i\approx y_i)\to(P(x_1,...,x_n)\leftrightarrow
P(y_1,...,y_n))\big)$.
\item[(5)]For each n-ary function symbol $f$,\\
$\forall x_1...\forall x_n\forall y_1...\forall y_n\big(\wedge_{i=1}^n(x_i\approx y_i)\to(f(x_1,...,x_n)\approx
f(y_1,...,y_n))\big)$.
\end{itemize}
Enforcing models to satisfy the congruence axioms leads to the right definition of structures
in first-order logic (that is the interpretation of an n-ary function symbol in
a model $\mathcal{M}$ would be a function $f^\mathcal{M}:M^n\to M$
and the interpretation of an n-ary predicate symbol $P$ would be a subset of $M^n$).

Now, let the $\mathcal{L}_d$-structure $\mathcal{M}$ satisfy the
congruence axioms for each function and predicate symbol, i.e,
\begin{center}
$\mathcal{M}\models\left\{\forall \bar{x}\forall \bar{y}(d(\bar{x},\bar{y})\to
d(f(\bar{x}),f(\bar{y})), \forall \bar{x}\forall \bar{y}(d(\bar{x},\bar{y})\to
(P(\bar{x})\leftrightarrow P(\bar{y}))\right\}_{f,P\in\mathcal{L}_d}$,
\end{center}
in which $d(\bar{x},\bar{y})$ is an abbreviation for $\wedge_{i=1}^n d(x_i,y_i)$. One
can easily verify that for each function symbol $f$,
$f^\mathcal{M}:(M^n,d^\mathcal{M})\to(M,d^\mathcal{M})$ would be a 1-Lipschitz continuous function,
in which $d^\mathcal{M}$ is a pseudo-ultrametric and
$d^\mathcal{M}(\bar{a},\bar{b})=\displaystyle\max_{1\le i\le n}\{d^\mathcal{M}(a_i,b_i)\}$
for every $\bar{a},\bar{b}\in M^n$. Furthermore, for each predicate symbol $P$,
$P^{\mathcal{M}}:(M^n,d^\mathcal{M})\to(V,d_{max})$ is a 1-Lipschitz continuous
function.
\begin{defn}
For a given language $\mathcal{L}_d$, containing a binary predicate symbol $d$, a \emph{Lipschitz
$\mathcal{L}_d$-structure} or simply a \emph{Lipschitz structure} $\mathcal{M}$ in the \g logic $\mathfrak{G}_V$
is a nonempty pseudo-ultrametric space $(M,d^\mathcal{M})$ called the universe of $\mathcal{M}$ together with:
\begin{itemize}
\item[a)] for any n-ary predicate symbol $P$ of $\mathcal{L}$, a 1-Lipschitz continuous function
\begin{center}
$P^{\mathcal{M}}:(M^n,d^\mathcal{M})\to(V,d_{max})$,
\end{center}
\item[b)] for any n-ary function symbol $f$ of $\mathcal{L}$, a 1-Lipschitz continuous function $f^{\mathcal{M}}:M^n\to M$,
\item[c)] for any constant symbol c of $\mathcal{L}$, an element $c^{\mathcal{M}}$ in the universe of $\mathcal{M}$.
\end{itemize}
\end{defn}
The following lemma is used to prove the compactness property for the \g logic
$\mathfrak{G}_\downarrow$. It is used in constructing a model for a theory with ultraproduct method.
\begin{lem}
Let $\mathcal{M}$ be a Lipschitz $\mathcal{L}_d$-structure. In the \g logic $\mathfrak{G}_{V,A}$ for every
$\mathcal{L}_d$-formula $\varphi(\bar{x})$ and $\bar{a},\bar{b}\subseteq M$,
\begin{center}
$d_{max}\big(\varphi^\mathcal{M}(\bar{a}),\varphi^\mathcal{M}(\bar{b})\big)\le d^\mathcal{M}(\bar{a},\bar{b})$.
\end{center}
\end{lem}
\begin{proof} Using the 1-Lipschitz continuity of
the interpretation of function and predicate symbols and also
1-Lipschitz continuity of $\dotto:(V^2,d_{max})\to (V,d_{max})$, the
proof is straightforward.
\end{proof}
The 1-Lipschitz continuity of the interpretation of function and predicate symbols
in Lipschitz structures leads to obtain a \emph{canonical ultrametric structure} $\mathcal{M}/d$
from a given Lipschitz structure $\mathcal{M}$. Verily, the underlying universe
of $\mathcal{M}/d$ is the set of equivalence classes of $M$ modulo the equivalence
relation $d^\mathcal{M}$. Furthermore, the interpretation of symbols of the language
could be defined by:
\begin{itemize}
\item $c^{\mathcal{M}/d}=[c^\mathcal{M}]_d$,
\item $f^{\mathcal{M}/d}([a_1]_d, ..., [a_n]_d)=f^\mathcal{M}(a_1, ..., a_n)$,
\item $P^{\mathcal{M}/d}([a_1]_d, ..., [a_n]_d)=P^\mathcal{M}(a_1, ..., a_n)$.
\end{itemize}
\begin{defn}
An $\mathcal{L}_d$-theory $T$ is called a \emph{metric-satisfiable} theory if there
is an ultramtric structure $\mathcal{M}$ which models $T$. Similarly, when $T$ has
a Lipschitz model $\mathcal{M}$, we call $T$ a \emph{Lipschitz-satisfiable} theory.
The notions of \emph{finitely metric-satisfiable} and \emph{finitely Lipschitz-satisfiable}
theories are similarly defined.
\end{defn}
\subsection{Ultraproduct method and the compactness theorem}\hfill

Let $\mathcal{L}$ be a first-order language (of any cardinality) and $T$ be a finitely
satisfiable $\mathcal{L}$-theory. Let $I=\mathcal{P}_{fin}(T)$ be the collection of all
finite subsets of $T$. For every $\varphi\in T$ assume that
$\varphi_T=\{\Sigma: \varphi\in\Sigma \text{ and } \Sigma\in I\}$. Obviously
$\{\varphi_T: \varphi\in T\}$ has the finite intersection property, and
so it is contained in an ultrafilter $\mathcal{D}$ on $I$.

Bellow we list some facts and notions about (ultra)filters on topological spaces.
\begin{itemize}
\item \cite{willard2004} A filter $\mathcal{F}$ on a topological space $X$ is convergent to
$x\in X$ if for all open sets $U$ containing $x$, $U\in\mathcal{F}$ .
\item \cite{willard2004} $X$ is compact Hausdorff space if and only if every ultrafilter
$\mathcal{F}$ on $X$ has a unique limit point.
\item Let $\{x_i\}_{i\in I}$ be a family of points of a topological space $X$. One could
view $\{x_i\}_{i\in I}$ as a function $f: I\to X$. If $\mathcal{D}$ is an (ultra)filter
on $I$, then $I_\mathcal{D}(X)=\{A\subseteq X: f^{-1}(A)\in\mathcal{D}\}$ is an (ultra)filter
on $X$. If $I_\mathcal{D}(X)$ is convergent to an element $x\in X$, we call $x$ a
$\mathcal{D}$-limit of the family $\{x_i\}_{i\in I}$.
\item Obviously $x$ is a $\mathcal{D}$-limit of $\{x_i\}_{i\in I}$ if and only if for
each open set $U$ containing $x$, the set $\{i\in I: x_i\in U\}$ belongs to the (ultra)filter $\mathcal{D}$.
\item If $X$ is a compact Hausdorff
and $x$ is the unique $\mathcal{D}$-limit of the family $\{x_i\}_{i\in I}$, we write $\lim_D x_i=x$.
\item If $X$ and $Y$ are two compact Hausdorff topological spaces, $f:X\to Y$ is a continuous function,
$\{x_i\}_{i\in I}$ is a family of elements of $X$, and $\mathcal{D}$ is an ultrafilter on $I$, then
$f(\lim_\mathcal{D}x_i)=\lim_\mathcal{D}f(x_i)$.
\end{itemize}
Now, we can prove the compactness theorem.
\begin{thm}\label{compactness final}
Let $\mathcal{L}_d$ be a first-order language and
$V'\subseteq\{0\}$. In the \g logic $\mathfrak{G}_V$ as well as the \g logic
$\mathfrak{G}_{V,V\setminus\{0,1\}}$, every finitely Lipschitz-satisfiable
theory $T$ is Lipschitz-satisfiable.
\end{thm}
\begin{proof}
Let $I=\mathcal{P}_{fin}(T)$ be the collection of all
finite subsets of $T$ and $\mathcal{D}$ be an ultrafilter on $I$ which
contain $\{\varphi_T: \varphi\in T\}$ where
$\varphi_T=\{\Sigma: \varphi\in\Sigma \text{ and } \Sigma\in I\}$.

Since $T$ is finitely Lipschitz-satisfiable, for each $T_i\in I$ there is a
Lipschitz structure $\mathcal{M}_i$ which models $T_i$. Let
$M=\prod_{i\in I}M_i$ and for each n-ary predicate symbol
$R$ define $R^\mathcal{M}:M^n\to V$ by
\begin{center}
$R^\mathcal{M}(\{x_{1i}\}_{i\in I}, ..., \{x_{ni}\}_{i\in I})
=\lim_\mathcal{D}R^{\mathcal{M}_i}(x_{1i}, ..., x_{ni})$.
\end{center}
Note that $V'\subseteq\{0\}$. So $(V,d_{max})$ is a compact Hausdorff space. Thus, the
$\mathcal{D}$-limit of the family $\{R^{\mathcal{M}_i}(x_{1i}, ..., x_{ni})\}_{i\in I}$
is unique and therefore $R^\mathcal{M}$ is well-defined.

One could easily verify that $d^\mathcal{M}$ is a pseudo-ultrametric on $M$.
Furthermore, Lipschitz continuity of $\{R^{\mathcal{M}_i}\}_{i\in I}$ implies
that $R^\mathcal{M}$ is Lipschitz continuous.

Now, for each constant symbol $c$ let $c^\mathcal{M}=\{c^{\mathcal{M}_i}\}_{i\in I}$. Also,
for each n-ary function symbol $f$ define $f^\mathcal{M}:M^n\to M$ by
$f^\mathcal{M}(\{x_{1i}\}_{i\in I}, ..., \{x_{ni}\}_{i\in I})=
\{f^{\mathcal{M}_i}(x_{1i}, ..., x_{ni})\}_{i\in I}$. Note that, by Lipschitz continuity of
$\{f^{\mathcal{M}_i}\}_{i\in I}$, $f^\mathcal{M}$ is well-defined and Lipschitz continuous.

Finally, by an induction on the complexity
of formulas and using the Lipschitz continuity of $\{R^{\mathcal{M}_i}\}_{i\in I}$
and also Lipschitz continuity of logical connectives on $(V,d_{max})$,
it is easy to see that for each formula $\varphi(x_1 ,..., x_n)$ and
elements $a_k=\{a_{ki}\}_{i\in I}$ of $M$ for $1\le k\le n$,
\begin{center}
$\varphi^\mathcal{M}(a_1,...,a_n)=\lim_\mathcal{D}\varphi^{\mathcal{M}_i}(a_{1i},...,a_{ni})$.
\end{center}
Thus, $\mathcal{M}$ is a Lipschitz model of $T$.
\end{proof}
\section{final remarks and further works}
We study the compactness property of extensions of first-order \g logic
with truth constants or with the Baaz's $\Delta$ connective.
The compactness theorem is one of the basic theorems that is used in the model theory
of classical first-order logic. Using the ultraproduct method, we
prove the compactness theorem for theories whose underlying
first-order language are uncountable. But,
the ultraproduct method forced us to consider
the \g sets with the reverse semantical meaning to interpret the similarity relation
as a (pseudo) metric. Thus, semantically $0$ is the absolute truth and $1$ is the
absolute falsity. The translation of results for the usual semantic of \g logic
could be stated as follows:
\begin{enumerate}
  \item (Corollary \ref{final compactness gn} and Corollary
  \ref{final entailment compactness gn}) If $A'=\emptyset$ then for any countable first-order language $\mathcal{L}$, the \g logic $\mathfrak{G}_{[0,1],A}$
  admits the compactness property as well as the entailment compactness.
  \item (Corollary \ref{compactness gn gdown}, Example \ref{fail entail compact gdown}, and
  Theorem \ref{approx entail com theorem}) If $A'=\{1\}$ then for any countable first-order language $\mathcal{L}$, the \g logic
  $\mathfrak{G}_{[0,1],A}$ admits the compactness property
  as well as the approximate entailment compactness, while the entailment compactness fails.
  \item (Example \ref{limit}) The compactness property fails in $\mathfrak{G}_{[0,1],A}$ whenever $A$ has a limit point
  $a\ne 1$ such that $a\in A\cup\{0\}$. In particular, RGL (first-order \g logic enriched with
  rational truth constants $A=(0,1)\cap\mathbb{Q}$ as nullary connectives) fails to have the
  compactness property.
  \item (Theorem \ref{compactness final}) If $V'\subseteq\{1\}$ then in the \g logics $\mathfrak{G}_{V}$ and
  $\mathfrak{G}_{V,V\setminus\{0,1\}}$ every finitely Lipschitz-satisfiable theory is Lipschitz-satisfiable.
  \item (Remark \ref{delta2}) If $A$ is finite then $\mathfrak{G}^\Delta_{[0,1],A}$ admit the compactness property as well as
  the entailment compactness property.
  \item (Example \ref{limit}, Example \ref{gdeltaconter1}) If $A'=\{1\}$ or $A'=\{0\}$ then the compactness fails in $\mathfrak{G}^\Delta_{[0,1],A}$.
\end{enumerate}
Regarding the first-order \g logic, it is seen that
the absolute truth and absolute falsity has an asymmetry.
Indeed, not falsity could be stated while it is impossible to
separate truth from not truth. The outcome of equipping the \g logic with
$\Delta$ is objective as a kind of symmetry. Indeed, not only one could
states "not truth" as well as the "not falsity", but also we have a
symmetry in items (5) and (6), and also the results are hold in both
semantical views of the \g logic.

A future interesting topic to study is the model theoretical aspects of
$\mathfrak{G}_\downarrow$, $\mathfrak{G}^*_\downarrow$ and
$\mathfrak{G}^\Delta_{[0,1]}$. Indeed, the expressive power of
the language of this logics for stating "$\mathcal{M}\nmodels\varphi$" by means of
"$\mathcal{M}\models\neg\Delta(\varphi)$" or "there is a natural number
n such that $\mathcal{M}\models\varphi\to\bar{\frac{1}{n}}$", helps us to
develop the model theory of these logics.

{\bf Acknowledgments:} Author is indebted to Massoud Pourmahdian for his enlightening
contributions through the preparation of this work.


\bibliographystyle{alpha}
\bibliography{amin}

\end{document}